\documentclass[12pt]{article}
\usepackage{mathrsfs}

\usepackage{t1enc}
\usepackage[latin1]{inputenc}
\usepackage[english]{babel}
\usepackage{amsmath,amsthm}
\usepackage{amsfonts}
\usepackage{latexsym}

\usepackage{graphicx}
\usepackage[natural]{xcolor}
\theoremstyle{plain}
\newtheorem{thm}{Theorem}
\newtheorem{lem}[thm]{Lemma}

\newtheorem{coro}[thm]{Corollary}

\theoremstyle{plain}

\theoremstyle{plain}

\theoremstyle{plain}

\title{Integer k-matching preclusion of graphs}
\author{Caibing Chang, Yan Liu\thanks{Corresponding author. \emph{E-mail address:} liuyan@scnu.edu.cn}\\
{\small School of Mathematical Sciences, South China Normal University, Guangzhou, 510631, China}}
\date{}

\begin{document}
\baselineskip 0.65cm

\maketitle
\begin{abstract}
As a generalization of matching preclusion number of a graph, we provide the (strong) integer $k$-matching preclusion number, abbreviated as $MP^{k}$ number ($SMP^{k}$ number), which is the minimum number of edges (vertices and edges) whose deletion results in a graph that has neither perfect integer $k$-matching nor almost perfect integer $k$-matching. In this paper, we show that when $k$ is even, the ($SMP^{k}$) $MP^{k}$ number is equal to the (strong) fractional matching preclusion number.
We obtain a necessary condition of graphs with an almost-perfect integer $k$-matching and a relational expression between the matching number and the integer $k$-matching number of bipartite graphs. Thus the $MP^{k}$ number and the $SMP^{k}$ number of complete graphs, bipartite graphs and arrangement graphs are obtained, respectively.

\end{abstract}

\section{Introduction}
\subsection{Definitions and Notations}
All graphs considered in this paper are undirected, finite and simple. Let $G$ be a graph. If $|V(G)|$ is odd (even), then we say that $G$ is odd (even). We refer to the book \cite{1BM} for graph theoretical notations and terminology that are not defined here. For $v\in V(G)$, we denote by $\Gamma(v)$ the set of edges incident with $v$. For two subsets $S,T$ of $V(G)$, let $E_{G}(S,T)$=$\{uv\in E(G)|u\in S, v\in T\}$. A vertex of degree 0 is called an isolated vertex. The number of isolated vertices of $G$ is denoted by $i(G)$. The number of odd components of $G$ is denoted by $c_{o}(G)$. The number of odd components with at least three vertices is denoted by $odd(G)$. Then $c_{o}(G)=i(G)+odd(G)$. A complete graph, a path and a cycle on $n$ vertices are denoted by $K_{n}$, $P_{n}$ and $C_{n}$, respectively.

A matching of $G$ is a subset of $E(G)$ in which no two edges are adjacent. The matching number of $G$, denoted by $\mu(G)$, is the maximum size of matchings. For a matching $M$, a vertex $v$ is $M$-saturated if $\Gamma(v)\cap M\neq \emptyset$, otherwise $v$ is $M$-unsaturated. A matching $M$ is perfect if every vertex is $M$-saturated. A matching $M$ is almost perfect if there exists exactly one  $M$-unsaturated vertex.
 A fractional matching is a function $f$: $E(G)\rightarrow [0,1]$ such that $\sum\limits_{e\in \Gamma(v)}f(e)\leq 1$ for any vertex $v$. Clearly, a fractional matching is a relaxation of matching.
 Let $k$ be a positive integer. An integer $k$-matching of a graph $G$ is a function $h$: $E(G)\rightarrow \{0,1,\ldots,k\}$ such that $\sum_{e\in \Gamma(v)}h(e)\leq k$ for any vertex $v\in V(G)$. Integer $k$-matching is a kind of generalization of matching. In fact, when $k=1$, integer $k$-matching is a matching. An edge is said to be 0-edge if $h(e)=0$.

 An integer $k$-matching is perfect if $\sum\limits_{e\in\Gamma(v)}h(e)=k$ for each vertex $v$, that is, $\sum\limits_{e\in E(G)}h(e)$ =$\frac{1}{2}\sum\limits_{v\in V(G)}$$\sum\limits_{e\in\Gamma(v)}$$h(e)=\frac{k|V(G)|}{2}$. It is clear that if $G$ has a perfect matching, then $G$ has a perfect integer $k$-matching which can be constructed by assigning $k$ to every edge in the perfect matching and assigning $0$ to other edges. But the opposite is not true. For example (see Fig.1), when $k$ is even, $C_{5}$ contains a perfect integer $k$-matching, but $C_{5}$ has no perfect matching. When $k$ is odd, $G_{0}$ contains a perfect integer $k$-matching, but $G_{0}$ has no perfect matching. An integer $k$-matching is almost-perfect \cite{YL} if there exists exactly one vertex $v'$ such that $\sum\limits_{e\in\Gamma(v')}h(e)=k-1$ and $\sum\limits_{e\in\Gamma(v)}h(e)=k$ for each other vertex $v$, that is, $\sum\limits_{e\in E(G)}h(e)$ =$\frac{1}{2}\sum\limits_{v\in V(G)}$$\sum\limits_{e\in\Gamma(v)}$$h(e)=\frac{k|V(G)|-1}{2}$. When $k$ is odd, it is clear that an odd cycle has an almost-perfect integer $k$-matching such that $v$ is unsaturated for any vertex $v$ of the odd cycle (see $C_7$ in Fig. 1).

\begin{figure}[h]
\flushleft
{
\begin{minipage}{9.5cm}
\begin{flushleft}
\includegraphics[scale=0.48]{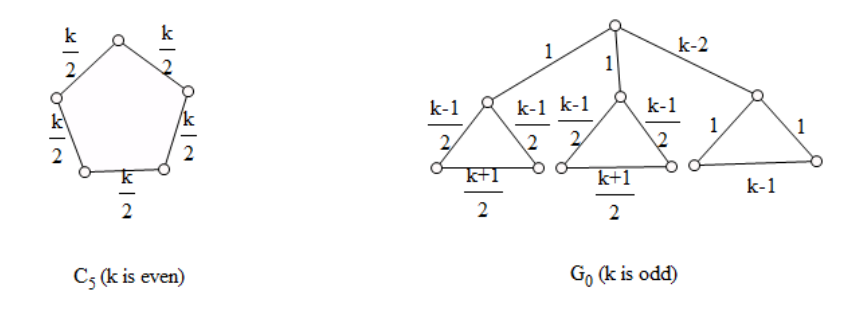}
\end{flushleft}
\end{minipage}
}
{
\begin{minipage}{6.3cm}
\begin{flushright}
\includegraphics[scale=0.55]{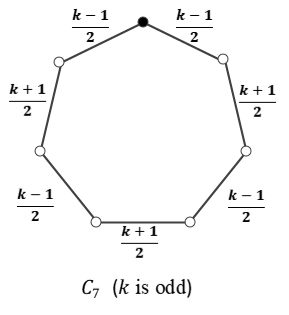}
\end{flushright}
\end{minipage}
}\\
\label{fig1}
~~~~~~~~~~~~~~~~~~~~~~~~~~~~~~~~~~~~~~~~~Fig. 1. (almost) perfect integer $k$-matching
\end{figure}

 Brigham et al \cite{5BHV} introduced the concept of matching preclusion, as a measure of robustness in the event of edge failure in interconnection networks. An edge subset $F$ of $G$ is a matching preclusion set ($MP$ set for short) if $G-F$ has neither perfect matching nor almost-perfect matching. The matching preclusion number of $G$, denoted by $mp(G)$, is the minimum size of $MP$ sets of $G$. An interconnection network with a large matching preclusion number may be considered as more robust in the event of link failures.
Park and Ihm \cite{4JHI} introduced the concept of strong matching preclusion, as a measure of robustness in the event of edge and vertices failure in interconnection networks. A set $F$ of edges and vertices of $G$ is a strong matching preclusion set ($SMP$ set for short) if $G-F$ has neither perfect matching nor almost-perfect matching. The strong matching preclusion number of $G$, denoted by $smp(G)$, is the minimum size of strong matching preclusion sets of G. The (strong) matching preclusion problems have been studied as interconnection network topologies. We refer the readers to \cite{1EL,1LJH,1HLZ, 4JHI,3EJJ} for details.

Liu and Liu \cite{5YWW} introduced the concept of fractional matching preclusion. An edge subset $F$ of $G$ is a fractional matching preclusion set if $G-F$ has no perfect fractional matching. The fractional matching preclusion number of $G$, denoted by $fmp(G)$, is the minimum size of fractional matching preclusion sets of $G$. Fractional matching preclusion is a generalization of matching preclusion. The fractional matching preclusion problems of many interconnection networks have been studied. We refer the readers to \cite{3TYE,4YCM} for details.

In this paper, we consider integer $k$-matching preclusion, which is another generalization of matching preclusion. Chang, Li and Liu \cite{5LCH} have obtained the integer $k$-matching preclusion of twisted cubes and $(n,s)$-star graphs.
\subsection{(Strong) Integer $k$-matching preclusion number}
An edge subset $F$ of $G$ is an $integer$ $k$-$matching$ $preclusion$ $set$ ($MP^{k}$ set for short) if $G-F$ has neither perfect integer $k$-matching nor almost perfect integer $k$-matching. The $integer$ $k$-$matching$ $preclusion$ $number$ ($MP^{k}$ number for short) of $G$, denoted by $mp^{k}(G)$, is the minimum size of $MP^{k}$ sets of $G$. Then $mp^{1}(G)$=$mp(G)$. For any $k\geq2$, $mp^{k}(G)\leq\delta(G)$ since the set of edges incident with $v$ is an $MP^{k}$ set for any vertex $v$.

A set $F$ of edges and vertices of $G$ is a $strong$ $integer$ $k$-$matching$ $preclusion$ $set$ ($SMP^{k}$ set for short) if $G-F$ has neither perfect integer $k$-matching nor almost perfect integer $k$-matching. The $strong$ $integer$ $k$-$matching$ $preclusion$ $number$ ($SMP^{k}$ number for short) of $G$, denoted by $smp^{k}(G)$, is the minimum size of $SMP^{k}$ sets of $G$. Then $smp^{1}(G)$=$smp(G)$. By the definition of $smp^{k}(G)$, we have that $smp^{k}(G)\leq mp^{k}(G)$.

In \cite{5YXL}, the $MP^{k}$ number and $SMP^{k}$ number of twisted cubes and (n, s)-star
graphs are given. In \cite{5YXL}, when $k$ is even, $G$ has a perfect fractional matching if and only if $G$ has a perfect integer $k$-matching. If $G$ has an almost perfect integer $k$-matching, say $h$, then $\sum\limits_{e\in E(G)}h(e)=\frac{k|V(G)|-1}{2}$ is an integer, which means that $k$ is odd and $|V(G)|$ is odd. So every graph has no almost perfect integer $k$-matching for even $k$. It follows that $fmp(G)=mp^{k}(G)$ and $fsmp(G)=smp^{k}(G)$ for even $k$. Hence we only consider the case that $k\geq3$ is odd in this paper. Thus $smp^{k}(G)\leq\delta(G)$ since the set of edges and vertices incident with $v$ or adjacent to $v$ is an $SMP^{k}$ set for any vertex $v$, $G$ is even if $G$ has a perfect integer $k$-matching and $G$ is odd if $G$ has an almost perfect integer $k$-matching.

\begin{figure}[h]\label{Fig:2}
\begin{center}
\includegraphics[scale=0.58]{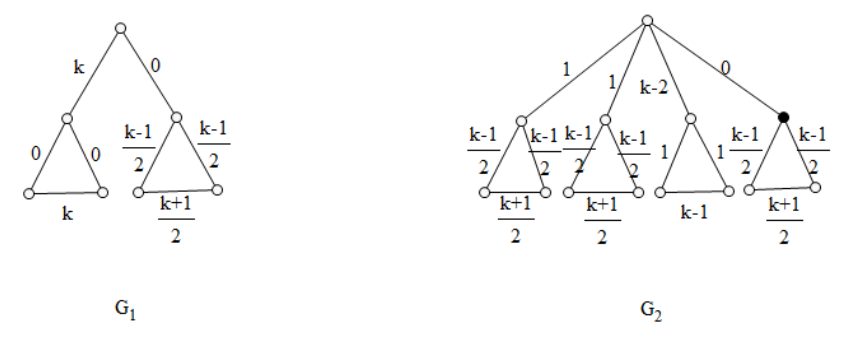}\\
Fig. 2. Examples of graphs
\end{center}
\end{figure}

When $G$ is even, $mp(G)\leq mp^{k}(G)$ by the fact that $G$ has a perfect integer $k$-matching if it has a perfect matching. But when $|V(G)|$ is odd, there is no relationship between $mp(G)$ and $mp^{k}(G)$. For example, $P_{3}$ and $G_{2}$ (see Fig. 2) are odd, $mp(P_{3})\geq 1>mp^{k}(P_{3})=0$, $mp(G_{2})=0<mp^{k}(G_{2})$.
Whatever $|V(G)|$ is even or odd, $smp(G)$ and $smp^{k}(G)$ have no a fixed relationship. For example, $P_{5}$, $G_{1}$ and $G_{2}$ (see Fig. 2) are odd, $smp(G_{1})=smp^{k}(G_{1})=1$, $smp(G_{2})=0<smp^{k}(G_{2})$ and $smp(P_{5})=1>smp^{k}(P_{5})=0$. Note that $C_{4}$ and $G_0$ (see Fig. 1) are even, $smp(C_{4})=2>smp^{k}(C_{4})=1$ and $smp(G_0)=0<smp_k(G_0)$.

Note that many interconnection networks $G$ have usually an even number of vertices and have the property that $mp(G)=\delta(G)$. So we have that $mp(G)=mp^{k}(G)=\delta(G)$. Since the arrangement graph $A_{n,s}$ has the property that $mp(A_{n,s})=\delta(A_{n,s})$ in the case that $2\leq s\leq n-2$, $mp^{k}(A_{n,s})=\delta(A_{n,s})$. So, in this paper, we investigate the $SMP^{k}$ number of arrangement graph $A_{n,s}$ in the case that $2\leq s\leq n-2$, and the $MP^{k}$ numbers and the $SMP^{k}$ number of $A_{n,1}$ and $A_{n,n-1}$, where $A_{n,1}$ is $K_n$ and $A_{n,n-1}$ is bipartite. The definition of arrangement graphs $A_{n,s}$ are given in the following.

\section{Arrangement graphs}

Let $n$ and $s$ be positive integers such that $n\geq2$ and $1\leq s\leq n-1$. The vertex set of arrangement graph $A_{n,s}$ is the set of all permutations on $s$ elements of the set $\{1,2,\ldots,n\}$. Two vertices corresponding to the permutations [$a_{1}$, $a_{2}$, \ldots, $a_{s}$] and [$b_{1}$, $b_{2}$, \ldots, $b_{s}$] are adjacent if and only if they differ in exactly one position. Then $A_{n,1}$ is the complete graph $K_n$. We know that $A_{n,n-1}$ is a bipartite graph (see \cite{8SBAD}) and $A_{n,n-2}$ is isomorphic to the alternating group graph $A_{n}$ (see \cite{8JSJS}). It can be seen easily that $A_{n,s}$ is $s(n-s)$-regular graph with $\frac{n!}{(n-s)!}$ vertices. For example, $A_{4,2}$ is shown in Fig. 3, where every vertex $[i,j]$ of $A_{n,2}$ is denoted by $ij$.
\begin{figure}[h]\label{Fig:3}
\begin{center}
\includegraphics[scale=0.23]{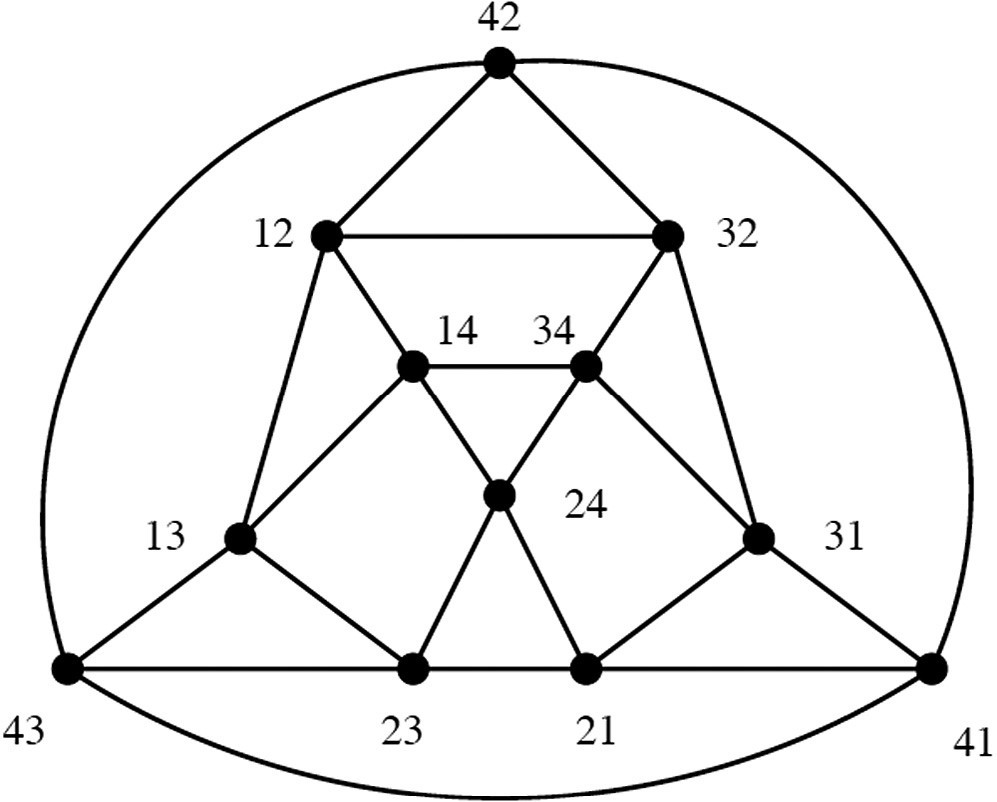}\\
Fig. 3. Arrangement graph $A_{4.2}$.
\end{center}
\end{figure}

The arrangement graphs were introduced in \cite{8KDA} to be a common generalization of star graphs and alternating group graphs, and to provide an even richer class of interconnection networks. Recently, researchers have found that they are excellent candidates as interconnection networks. For more results of arrangement graphs, we refer the readers to \cite{8ECM, 8ECO, 8HCHT}.

Let $V_{i}$ be the set of vertices representing permutations whose $s$th element is $i$ and $H_{i}$ the subgraph of $A_{n,s}$ induced by $V_{i}$ for each $i\in\{1,\cdots,n\}$. Then $H_{i}$ can be seen the arrangement graph defined on permutations of $s-1$ elements of $\{1,2, \ldots, n\}-\{i\}$. Hence $H_{i}$ is isomorphic to $A_{n-1,s-1}$. Thus $H_{i}$ has $\frac{(n-1)!}{(n-s)!}$ vertices for each $1\leq i\leq n$ and $A_{n,s}$ can be decomposed into $n$ subgraphs $H_{i}$. The edges whose two end-vertices belong to different $H_{j}$ are called $cross$ edges. Then there are exactly $\frac{(n-2)!}{(n-s-1)!}$ cross edges between $H_{i}$ and $H_{j}$ for any $1\leq i<j \leq n$ which are not adjacent to each other.
Each vertex [$a_{1}$, $a_{2}$, \ldots, $a_{s-1},\ i $] of $H_{i}$ has exactly $n-s$ neighbors in $V(A_{n,s})-V(H_{i})$, which are $[a_{1}, a_{2}, \ldots, a_{s-1},\ j ]\ (j\in \{1,2,\cdots,n\}-\{a_{1}, a_{2}, \ldots, a_{s-1},\ i\})$ and then they belong to $n-s$ different subgraphs $H_j'$s, called the $outer$ neighbors of [$a_{1}$, $a_{2}$, \ldots, $a_{s-1}$, $i$]. The subgraph $H_i$ is called adjacent to a vertex $u$ or $u$ is adjacent to $H_i$ if $u$ has an outer neighbor in $V(H_i)$. Then there exist exactly $n-s$ subgraphs among $H_1,\cdots,H_{i-1},H_{i+1},\cdots,H_n$ adjacent to $v$ for any vertex $v$ of $H_i$.

\section{Main results}

Note that if $G-F$ has a perfect integer $k$-matching or an almost perfect integer $k$-matching for any $F\subseteq V(G)\cup E(G)$ such that $|F|=t$, then $G-F_{1}$ has also a perfect integer $k$-matching or an almost perfect integer $k$-matching for any $F_{1}\subseteq V(G)\cup E(G)$ such that $|F_{1}|<t$. So $smp^{k}(G)\geq t+1$ if and only if $G-F$ has a perfect integer $k$-matching or an almost perfect integer $k$-matching for any $F\subseteq V(G)\cup E(G)$ such that $|F|=t$.

A Hamiltonian graph $G$ is called $k$-$fault \ Hamiltonian$ if $G-F$ remains Hamiltonian for every $F\subset V(G)\cup E(G)$ with that $|F|\leq k$. A graph is $Hamiltonian$ $connected$ if there is a Hamiltonian path between any two vertices. A Hamiltonian connected graph $G$ is called $k$-$fault \ Hamiltonian \ connected$ if $G-F$ remains Hamiltonian connected for every $F\subset V(G)\cup E(G)$ with that $|F|\leq k$. The following lemmas are useful.

\begin{lem}[\cite{5CYJLH}]\label{lem:3.10}
Let $n\geq 4$ be an integer. Then $K_{n}$ is $(n-3)$-fault Hamiltonian and $(n-4)$-fault Hamiltonian connected.
\end{lem}

\begin{lem}[\cite{5BHV}]\label{Pro:2}
For any even integer $n\geq2$, $mp(K_{n})=n-1$ and for any odd integer $n\geq11$, $mp(K_{n})=2n-3$ .
\end{lem}

\begin{lem}[\cite{5JSF}]\label{lem:3.8}
Let $n\geq 6$ and $F\subseteq V(K_n)\cup E(K_n)$ such that $|F|\leq n-2$ and $\delta(K_n-F)\geq2$. Then $K_{n}-F$ is Hamiltonian.
\end{lem}

\begin{lem}[\cite{8HCHT}]\label{Thm:1.8}
Let $n$ and $s$ be two integers with that $1\leq s\leq n-2$ and $F\subseteq V(A_{n,s})\cup E(A_{n,s})$. If $|F|\leq s(n-s)-2$, then $A_{n,s}-F$ is Hamiltonian.
\end{lem}

Based on Lemma \ref{Thm:1.8}, Cheng et al. investigated the matching preclusion number and strong matching preclusion number of arrangement graphs $A_{n,s}$ as follows.
\begin{lem}[\cite{8ECL}]\label{Thm:1.9}
Let $n$ and $s$ be integers with that $2\leq s\leq n-3$. Then $mp(A_{n,s})=s(n-s)$.
\end{lem}

\begin{lem}[\cite{8ECO}]\label{Thm:1.10}
Let $n$ and $s$ be integers with that $2\leq s\leq n-2$. Then $smp(A_{n,s})=s(n-s)$.
\end{lem}

Since $smp(A_{n,n-2})\leq mp(A_{n,n-2})\leq \delta(A_{n,n-2})$, $mp(A_{n,n-2})=2(n-2)$ by Lemma \ref{Thm:1.10}.

In the following, we study the (strong) integer $k$-matching preclusion number of $A_{n,s}$ by distinguish three cases. Note that $k\geq 3$ is odd.

\subsection{$MP^{k}$ and $SMP^{k}$ numbers of  $K_n$}

Since $A_{n,1}$ is a complete graph $K_n$, we study the (strong) integer $k$-matching preclusion number of complete graphs. Firstly, we need to prove the following lemma.
\begin{lem}\label{lem:4.1}
Let $G$ be a graph with an almost perfect integer $k$-matching. Then $odd(G-S)+k\cdot i(G-S)\leq k|S|+1$ for any $S\subseteq V(G)$.
\end{lem}

\begin{proof} Let $h$ be an almost perfect integer $k$-matching of $G$ and $u\in V(G)$ such that $\sum\limits_{e\in \Gamma(u)}h(e)=k-1$ and $\sum\limits_{e\in \Gamma(v)}h(e)=k$ for any $v\in V(G)-\{u\}$. Let $S\subseteq V(G)$, $Q_{1},\cdots , Q_{t}$ be the odd non-trivial components, $\{v_{1}, \cdots , v_{l}\}$ be the isolated vertex set and $H$ be the union of even components of $G-S$, respectively. We distinguish the following cases.

\noindent\textbf{Case 1.}\ There exists a $Q_j$ such that $u\in V(Q_{j})$.
Then we have that $\sum\limits_{e\in E(\overline{S},S)}h(e)\leq k|S|$, $\sum\limits_{e\in E(v_{i},S)}h(e)=k$ for any $i\in \{1, \ldots , l\}$ and $\sum\limits_{e\in E(V(Q_{i}),S)}h(e)\geq 1$ for any $i\neq j$ since both $Q_{i}$ and $k$ are odd. Hence
$$odd(G-S)+k\cdot i(G-S)\leq \sum\limits_{i\in \{1,\cdots,t\}-\{j\}} \sum\limits_{e\in E(V(Q_{i}),S)}h(e)+1+\sum\limits_{e\in E(\{v_{1},\cdots,v_{l}\},S)}h(e)$$
$$\ \ \ \ \leq\sum\limits_{e\in E(\overline{S},S)}h(e)+1\leq k|S|+1.$$

\noindent\textbf{Case 2.}\ There exists $j\in \{1, \ldots , l\}$ such that $u=v_{j}$. Similar to the statement above, we have that
$$ odd(G-S)+k\cdot i(G-S)\leq \sum \limits_{i=1}^{t}\sum\limits_{e\in E(V(Q_{i}),S)}h(e)+\sum\limits_{i\in \{1,\cdots,l\}-\{j\}} \sum\limits_{e\in E(v_{i},S)}h(e)+\sum\limits_{e\in E(v_{j},S)}h(e)+1$$
  $$\leq \sum\limits_{e\in E(\overline{S},S)}h(e)+1\leq k|S|+1. \ \ \ \ \ \ \ \ \ \ \ \ \ $$

\noindent\textbf{Case 3.}\ $u\in V(H)$. By the similar reasons as above, we have that
$$odd(G-S)+k\cdot i(G-S)\leq \sum \limits_{i=1}^{t}\sum\limits_{e\in E(V(Q_{i}),S)}h(e)+\sum\limits_{e\in E(\{v_{1},\cdots,v_{l}\},S)}h(e)\leq\sum\limits_{e\in E(\overline{S},S)}h(e)\leq k|S|.$$

\noindent\textbf{Case 4.}\ $u\in S$.
Then we have that $\sum\limits_{e\in E(\overline{S},S)}h(e)\leq k|S|-1$. Thus
$$odd(G-S)+k\cdot i(G-S)\leq \sum \limits_{i=1}^{t}\sum\limits_{e\in E(V(Q_{i}),S)}h(e)+\sum\limits_{e\in E(\{v_{1},\cdots,v_{l}\},S)}h(e)\leq\sum\limits_{e\in E(\overline{S},S)}h(e)\leq k|S|-1.$$
\end{proof}

\begin{thm}\label{Thm:7}
For any complete graph $K_n$,
$mp^{k}(K_{n})=
 \begin{cases}
 n-2 ,& \text{$n=3,5$}.\\
n-1 ,& \text{otherwise}.
\end{cases}
$

\end{thm}

\begin{proof} Since $K_{3}$ has an almost perfect integer $k$-matching, $mp^{k}(K_{3})\geq 1$. But $P_3=K_3-e$ has no almost perfect integer $k$-matching. Thus $mp^{k}(K_{3})\leq1$. Hence $mp^{k}(K_{3})=1$.

Since $mp^{k}(K_{4})\leq \delta (K_{4})=3$ and $mp^{k}(K_{4})\geq mp(K_{4})=3$ by Lemma~\ref{Pro:2}, we have that $mp^{k}(K_{4})=3$.

Clearly, $K_{5}-F$ has a Hamiltonian cycle for any $F\subseteq E(K_{5})$ with that $|F|=2$. So $K_{5}-F$ has an almost perfect integer $k$-matching. Thus $mp^{k}(K_{5})\geq 3$. Let $F=\{uv, uw, vw\}\subset E(K_{5})$ and $S=\{x,y\}\subset V(K_{5})$ (see Fig. 4). Then $odd(K_{5}-F-S)+k\cdot i(K_{5}-F-S)=3k>2k+1$. So $K_{5}-F$ has no almost perfect integer $k$-matching by Lemma~\ref{lem:4.1}. Thus $mp^{k}(K_{5})\leq 3$. Hence $mp^{k}(K_{5})=3$.
\begin{figure}\label{Fig:6}
\begin{center}
\includegraphics[scale=0.35]{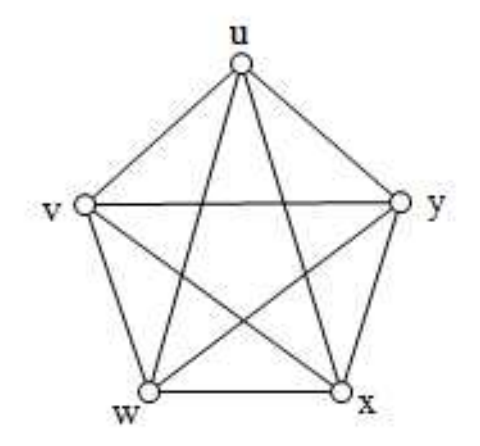}\\
Fig. 4. $K_{5}$
\end{center}
\end{figure}
Let $n\geq 6$. It clear that $mp^{k}(K_{n})\leq \delta (K_{n})=n-1$. Let $F\subseteq E(K_{n})$ such that $|F|=n-2$. We prove that $F$ is not $MP_k$ set by the following cases.

\noindent\textbf{Case 1.}\ $\delta (K_{n}-F)\geq2$.
Then $K_{n}-F$ is Hamiltonian by Lemma~\ref{lem:3.8}. Thus $K_{n}-F$ has a perfect integer $k$ matching or an almost perfect integer $k$-matching.

\noindent\textbf{Case 2.}\ $\delta (K_{n}-F)=1$.
Suppose that $d_{K_{n}-F}(v)=1$ and $uv\in E(K_{n}-F)$. It follows that every edge in $F$ is incident with $v$. Since $K_{n}-F-\{u,v\}=K_{n-2}$, we can construct a perfect integer $k$-matching or an almost perfect integer $k$-matching of $K_{n}-F$ which assigns $k$ to $uv$.

Hence $mp_k(K_n)\geq n-1$. Thus $mp_k(K_n)= n-1$ for $n\geq6$.
\end{proof}

Next, we investigate the strong integer $k$-matching preclusion number of complete graphs.
\begin{thm}\label{Thm:9}
For any complete graph $K_n$,
$smp^{k}(K_{n})=
 \begin{cases}
 n-2 ,& \text{$n=3,4,5$}.\\
n-1 ,& \text{otherwise}.

\end{cases}
$

\end{thm}

\begin{proof}
Since $K_{3}$ has an almost perfect integer $k$-matching, $smp^{k}(K_{3})\geq 1$. And $smp^{k}(K_3)\leq mp^{k}(K_3)=1$ by Theorem~\ref{Thm:7}. Thus $smp^{k}(K_{3})=1$.

Let $v\in V(K_{4})$ and $e\in E(K_{4}-v)$. Then $K_{4}-\{v,e\}$ has neither perfect integer $k$-matching nor almost perfect integer $k$-matching. Thus $smp^{k}(K_{4})\leq 2$. Since $K_{4}-\{u\}$ has an almost perfect integer $k$-matching for any $u\in V(K_{4})$ and $K_{4}-\{e\}$ has a perfect integer $k$-matching for any $e\in E(K_{4})$, $smp^{k}(K_{4})\geq 2$. Thus $smp^{k}(K_{4})=2$.

Clearly, $K_{5}-F$ is Hamiltonian for any $F\subset E(K_{5})\cup V(K_{5})$ with that $|F|=2$. So $K_{5}-F$ has a perfect integer $k$-matching or an almost perfect integer $k$-matching. Thus $smp^{k}(K_{5})\geq 3$. And $smp^{k}(K_5)\leq mp^{k}(K_5)=3$ by Theorem~\ref{Thm:7}. Hence $smp^{k}(K_{5})=3$.

Let $n\geq 6$ and $F\subseteq E(K_{n})\cup V(K_{n})$ such that $|F|=n-2$. Clearly $smp_k(K_n)\leq n-1$. We distinguish the following cases to prove that $F$ is not $SMP_k$ set.

\noindent\textbf{Case 1.}\ $\delta (K_{n}-F)\geq2$.
Then $K_{n}-F$ is Hamiltonian by Lemma~\ref{lem:3.8}. Thus $K_{n}-F$ has a perfect integer $k$-matching or an almost perfect integer $k$-matching.

\noindent\textbf{Case 2.}\ $\delta (K_{n}-F)=1$.
Then we can assume that $d_{K_{n}-F}(v)=1$ and $uv\in E(K_{n}-F)$. It follows that every edge in $F$ is incident with $v$ and every vertex in $F$ is adjacent to $v$. Since $K_{n}-F-\{u,v\}$ is a complete graph, we can construct a perfect integer $k$-matching or an almost perfect integer $k$-matching of $K_{n}-F$ which assigns $k$ to $uv$.

Hence $smp_k(K_n)\geq n-1$. Thus $smp_k(K_n)= n-1$ for $n\geq6$.
\end{proof}

\subsection{$MP^{k}$ and $SMP^{k}$ numbers of bipartite graphs}

Note that $A_{n,n-1}$ is a bipartite graph. A connected graph $G$ is said to be an $odd$-$cycle$-$tree$ $graph$ if every block of $G$ is either an odd cycle or $K_{2}$, $\delta(G)\geq 2$ and no two odd cycles have a common vertex (see Fig. 5).
\begin{lem}[\cite{5YXL}]\label{lem:1.1}
Let $k$ be an integer, $h$ a maximum integer $k$-matching of a graph $G$ with the maximum number of 0-edges and $H$ the subgraph of $G$ induced by the set of edges $e$ with that $h(e)\neq 0$. Then every component of $H$ is a single edge or an odd-cycle-tree graph.
\end{lem}

Let $G$ be a bipartite graph and $h$ a maximum integer $k$-matching of $G$ with the maximum number of 0-edges. Then every component of $H$ is a single edge by Lemma~\ref{lem:1.1}. So $E(H)$ is a matching. Then $\mu_{k}(G)\leq k\mu(G)$, where $\mu_{k}(G)$, named integer $k$-matching number of $G$, is the maximum value of $\sum\limits_{e\in E(G)}h(e)$ over all integer $k$-matching $h$. On the other hand, for a maximum matching $M$, we can obtain an integer $k$-matching by assigning $k$ to each edge in $M$ and $0$ to other edges. So $\mu_{k}(G)\geq k\mu(G)$. Then we have the following corollary.

\begin{coro}\label{cor:1}
If $G$ is a bipartite graph, then $\mu_{k}(G)=k\mu(G)$.
\end{coro}

By the above corollary, we have the following corollary.
\begin{coro}\label{cor:2}
Let $G$ be a bipartite graph. Then $G$ has a perfect integer $k$-matching if and only if $G$ has a perfect matching.
\end{coro}

Obviously, if $G$ is a bipartite graph with the odd number of vertices, then $\mu(G)\leq\frac{|V(G)|-1}{2}$. By Corollary~\ref{cor:1}, $\mu_{k}(G)=k\mu(G)\leq\frac{k|V(G)|-k}{2}<\frac{k|V(G)|-1}{2}$. So we get the following corollary.

\begin{coro}\label{cor:3}
Let $G$ be a bipartite graph with the odd number of vertices. Then $G$ has no almost perfect integer $k$-matching.
\end{coro}
\begin{figure}\label{Fig:5}
\begin{center}
\includegraphics[scale=0.345]{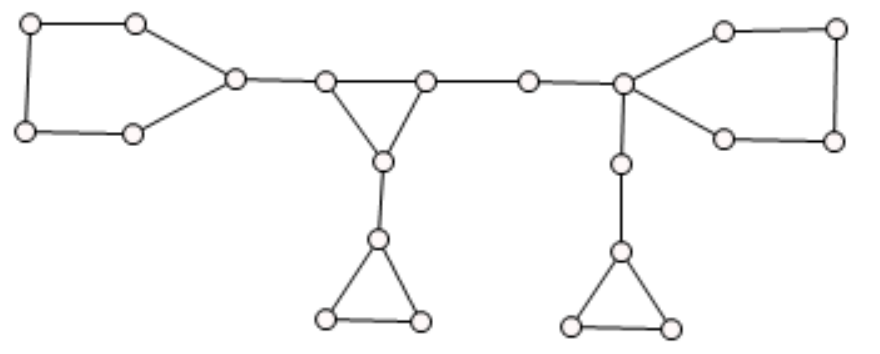}\\
Fig. 5. Odd-cycle-tree graph
\end{center}
\end{figure}
\begin{thm}\label{cor:4}
Let $G$ be a bipartite graph. Then $mp^{k}(G)=smp^{k}(G)=0$ if $G$ is odd. Otherwise, $mp^{k}(G)=mp(G)$ and $smp^{k}(G)\leq smp(G)$.
\end{thm}

\begin{proof} Let $G$ is odd. By Corollary~\ref{cor:2} and Corollary~\ref{cor:3}, $G$ contains neither perfect integer $k$-matching nor almost-perfect integer $k$-matching. Thus $mp^{k}(G)=smp^{k}(G)=0$. Let $G$ be even. By Corollary~\ref{cor:2}, $mp^{k}(G)=mp(G)$. Let $F$ be a minimum $SMP$ set of $G$. Then $|F|=smp(G)$ and $G-F$ contains neither perfect matching nor almost perfect mathing. By Corollary~\ref{cor:1} and Corollary~\ref{cor:2}, $G-F$ contains no perfect integer $k$-matching and almost-perfect integer $k$-matching. So $F$ is a strong integer $k$-matching preclusion set. Thus $smp^{k}(G)\leq smp(G)$.
\end{proof}

Note that the equivalence $"smp^{k}(G)= smp(G)"$ is not true. For example, $smp^{k}(C_4)=1$ and $smp(C_4)=2$.

\subsection{$SMP^{k}$ numbers of $A_{n,s}$ ($2\leq s\leq n-2$)}
we investigate the $SMP^{k}$ number of $A_{n,s}$ in the following two cases that $3\leq s\leq n-2$ and $s=2$.

\begin{thm}\label{Th3}
Let $n$ and $s$ be integers such that $3\leq s\leq n-2$. Then $smp^{k}(A_{n,s})=s(n-s)$.
\end{thm}

\begin{proof}
We have that $smp^{k}(A_{n,s})\leq \delta(A_{n,s})=s(n-s)$. Next we prove that $smp^{k}(A_{n,s})\geq s(n-s)$. It suffices to prove that $A_{n,s}-F$ has a perfect integer $k$-matching or an almost perfect integer $k$-matching for any $F\subseteq V(A_{n,s})\cup E(A_{n,s})$ with that $|F|=s(n-s)-1$. Let $F_{i}=F\cap (V(H_{i})\cup E(H_{i}))$ and $G_{i}=H_{i}-F_{i}$ for any $i\in\{1,\ 2, \cdots, n\}$. Then $F-\cup_{i=1}^{n} F_i$ consists of cross edges in $A_{n,s}$. Without loss of generality, suppose that $|F_{1}| \geq |F_{2}|\geq \cdots \geq |F_{n}|$.

If $A_{n,s}-F$ is even, then $A_{n,s}-F$ has a perfect matching since $smp(A_{n,s})=s(n-s)$ by Lemma \ref{Thm:1.10}. It follows that $A_{n,s}-F$ has a perfect integer $k$-matching.
Suppose that $A_{n,s}-F$ is odd. Then the number of odd $G_{i}$ among $G_{1}, G_{2},\cdots,G_{n}$ is odd. We distinguish the following cases to prove that $A_{n,s}-F$ has an almost perfect integer $k$-matching.

\noindent\textbf{Case 1.}\ $|F_{1}|\leq (s-1)(n-s)-2$.

Then $|F_i|\leq |F_{1}|\leq (s-1)(n-s)-2$ for any $i\geq2$. By Lemma \ref{Thm:1.8}, $G_i$ ($i\geq1$) is Hamiltonian. Then every even $G_i$ has a perfect integer $k$-matching.
If there exists exactly one odd subgraph $G_i$ among $G_1,G_2,\cdots,G_n$, then $G_i$ has an almost perfect integer $k$-matching. So $A_{n,s}-F$ has an almost perfect integer $k$-matching. Suppose that there exists at least three odd subgraphs. For any two odd subgraphs $G_{i}$ and $G_{j}$,
$|E_{A_{n,s}-F}(V(G_{i}),V(G_{j}))|\geq |E_{A_{n,s}}(V(H_{i}),V(H_{j}))|-|F|= \frac{(n-2)!}{(n-s-1)!}-(s(n-s)-1)\geq 1$ according to the property that the cross edges in $E_{A_{n,s}}(V(H_{i}),V(H_{j}))$ are independent.
Then we can find an even Hamiltonian path of the subgraph induced by $V(G_i)\cup V(G_j)$ in $A_{n,s}-F$. Hence there exists a perfect integer $k$-matching of the subgraph induced by $V(G_i)\cup V(G_j)$. It follows that $A_{n,s}-F$ has an almost perfect integer $k$-matching.

\noindent\textbf{Case 2.}\ $|F_{1}|=(s-1)(n-s)-1$.

Then $|F_{i}|\leq |F-F_{1}|= n-s\leq (s-1)(n-s)-2$. By Lemma \ref{Thm:1.8}, $G_{i}\ (i\geq2)$ is Hamiltonian and $G_1$ has a Hamiltonian path, say $P$.

If $G_{1}$ is even, then $G_1$ has a perfect integer $k$-matching. For any two subgraphs $G_{i}$ and $G_{j}$ ($2\leq i<j\leq n$), $|E_{A_{n,s}-F}(V(G_{i}),V(G_{j}))|\geq |E_{A_{n,s}}(V(H_{i}),V(H_{j}))|-|F-F_{1}|=\frac{(n-2)!}{(n-s-1)!}-(n-s)\geq 1.$ By the similar method in Case 1 as above, we can find an almost perfect integer $k$-matching of $A_{n,s}-F$.

Suppose that $G_{1}$ is odd. Let $u$ and $v$ be two end-vertices of $P$. Since $|F-F_{1}|= n-s$ and $u$ and $v$ have $2(n-s)$ outer neighbors in $A_{n,s}$, we can assume that $u$ has an outer neighbor in $A_{n,s}-F$. Since $u$ is adjacent to $n-s$ subgraphs and $|F-F_{1}|= n-s$, there is a subgraph $G_t$ adjacent to $u$ such that $|F_t|\leq 1$.
Suppose that $w\in V(G_t)$ such that $uw\in E(A_{n,s}-F)$. Then the subgraph induced by $V(G_1)\cup \{w\}$ of $A_{n,s}-F$ has a perfect integer $k$-matching. Let $F_t'=F_t\cup \{w\}$. Then $|F_t'|\leq2$. By Lemma \ref{Thm:1.8}, $G_t'=H_t-F_t'$ is Hamiltonian. Similarly, $|E_{A_{n,s}-F}(V(G_{i}),V(G_{j}))|\geq1$ for any pair of subgraphs $G_i$ and $G_j$ among $G_2,\cdots,G_t',\cdots,G_n$. Then we can obtain an almost perfect integer $k$-matching of $A_{n,s}-F$ by the similar method in Case 1 as above.

\noindent\textbf{Case 3.}\ $|F_{1}|=(s-1)(n-s)+\alpha$, where $0\leq \alpha \leq n-s-1$.

We can choose a subset $S$ of $F_1$ such that $|S|=\alpha+1$ and $F_1-S$ has even number of vertices. Let $F_{1}'=F_{1}-S$. Then $|F_{1}'|=(s-1)(n-s)-1$ and $H_{1}-F_{1}'$ is even. By Lemma \ref{Thm:1.10}, $H_{1}-F_{1}'$ has a perfect matching, say $M$. Let $A=S\cap V(H_1)=\{v_{1},v_{2},\cdots,v_{t}\}$, $B=S\cap E(H_1)=\{w_{1}s_{1},w_{2}s_{2},\cdots,w_{r}s_{r}\}$ and $T$ be the set of all $M$-unsaturated vertices of $G_1$. Since $G_1=H_1-F_1=(H_1-F_1')-S$, we can assume that $T=\{z_1,\cdots,z_{t'},w_{1},\cdots,w_{r'},s_{1},\cdots,s_{r'}\}$, where $v_1z_1,\cdots,v_{t'}z_{t'}\in M_1$. Then $0\leq t'\leq t$, $0\leq r'\leq r$. Thus $|T|=t'+2r'\leq 2(\alpha+1)$.

Since $|F-F_{1}|=n-s-1-\alpha\leq n-s-1$ and each vertex has $n-s$ outer neighbors in $A_{n,s}$, each vertex in $T$ has at least one outer neighbor in $A_{n,s}-F$.
Let $T'=\{z_{1}',\cdots,z_{t'}',w_{1}',\cdots,w_{r'}',s_{1}',\cdots,s_{r'}'\}$, where $z_{i}'$, $w_{j}'$ and $s_{j}'$ are the outer neighbor of $z_{i}$, $w_{j}$ and $s_{j}$ in $T$, respectively. Then $|T'|=|T|\leq2(\alpha+1)$ and the subgraph induced by $V(G_1)\cup T'$ has a perfect matching. Thus the subgraph induced by $V(G_1)\cup T'$ has a perfect integer $k$-matching.
Furthermore, in $A_{n,s}-F$, there exists at most one vertex in $T$ adjacent only to $G_2$ since any pair of vertices of $G_1$ have not common outer neighbors and $|F-F_{1}|\leq n-s-1$. Thus $|T'\cap V(G_2)|\leq 1$. Let $F_{i}'=F_i\cup (T'\cap V(G_i))$  and $G_i'=H_i-F_i'$ for any $i\geq2$.

Now we consider the case that $|T|\geq n-s-1$ and each vertex in $T$ has at least $n-s-1$ outer neighbors in $A_{n,s}-F$. Then we can find $n-s-1$ vertices in $T$, say $\{u_1,\cdots,u_{n-s-1}\}$, such that the outer neighbors $u_1',\cdots,u_{n-s-1}'$ of $ u_1,\cdots,u_{n-s-1}$ belong to $n-s-1$ distinct subgraphs among $G_2,\cdots,G_n$, respectively, and $|(\{u_1',\cdots,u_{n-s-1}'\}\cup T')\cap V(G_2)|\leq1$. Replace the outer neighbors in $T'$ of $\{u_1,\cdots,u_{n-s-1}\}$ with $\{u_1',\cdots,u_{n-s-1}'\}$, the resulting set is still denoted by $T'$. Hence there are at least $n-s-1$ subgraphs $G_i$ $(i\geq2)$ such that $V(G_i)\cap T'\neq \emptyset$ and $|T'\cap V(G_2)|\leq 1$ and the subgraph induced by $V(G_1)\cup T'$ has also a perfect integer $k$-matching.

\textbf{Claim.} $|F_i'|\leq (s-1)(n-s)-2$.

If $0\leq\alpha\leq n-s-3$ or $|T'|\leq 2\alpha$, then $|F_i'|=|F_i\cup (T'\cap V(G_i))|\leq |F-F_1|+|T'|\leq2(n-s)-2\leq(s-1)(n-s)-2$.
Suppose that $n-s-2\leq\alpha\leq n-s-1$ and $|T'|\geq 2\alpha+1$. Then $|T|=|T'|\geq n-s-1$, $|F_2|\leq|F-F_1|=n-s-1-\alpha\leq1$ and $|F_i|=0$ ($i\geq3$).
Since $|T'\cap V(G_2)|\leq 1$, $|F_2'|=|F_2\cup (T'\cap V(G_2))|=|F_2|+|T'\cap V(G_2)|\leq2\leq(s-1)(n-s)-2$.
Suppose that $i\geq3$. Then $F_i'=T'\cap V(G_i)$. If $|T'|\leq2$, then $|F_i'|= |T'\cap V(G_i)|\leq|T'|=2\leq(s-1)(n-s)-2$. Suppose that $|T'|\geq3$. Since $|F-F_1|\leq1$ and each vertex of $A_{n,s}$ has $n-s$ outer neighbors, each vertex in $T$ has at least $n-s-1$ outer neighbors in $A_{n,s}-F$. In this case, there are at least $n-s-1$ subgraphs $H_i$ $(i\geq2)$ such that $V(H_i)\cap T'\neq \emptyset$.
Thus $|F_i'|=|T'\cap V(G_i)|\leq |T'|-(n-s-2)\leq2(\alpha+1)-(n-s-2)\leq2(n-s)-2\leq(s-1)(n-s)-2$. Thus the claim is true.

By the claim and Lemma \ref{Thm:1.8}, $G_i'$ ($i\geq 2$) is Hamiltonian.
For any two subgraphs $G_{i}'$ and $G_{j}'$ ($2\leq i<j\leq n$), $|E_{A_{n,s}-F}(V(G_{i}'),V(G_{j}'))|\geq |E_{A_{n,s}}(V(H_{i}),V(H_{j}))|-|(F-F_{1})\cup T'|\geq \frac{(n-2)!}{(n-s-1)!}-2(n-s)\geq 1$.
Hence we can obtain an almost perfect integer $k$-matching of the subgraph induced by $V(G_{2}')\cup \cdots \cup V(G_{n}')$ in $A_{n,s}-F$. Thus $A_{n,s}-F$ has an almost perfect integer $k$-matching.
\end{proof}

Next, we consider the case that $s=2$. Then a vertex of $A_{n,2}$ is denoted by $ij$, where $i,j\in \{1,2,\cdots,n\}$. For convenience, the notations $H_i$, $F_i$ and $G_i$ in the proof of Theorem \ref{Th3} are used in the following. In this case, $H_i=K_{n-1}$. According to the definition of $A_{n,s}$, there are exactly $n-2$ cross edges between $H_{i}$ and $H_{j}$ in $A_{n,2}$ which are independent.
Each vertex in $A_{n,2}$ have $n-2$ outer neighbors which belong to $n-2$ different subgraphs. Moreover, any pair of vertices in $V(H_i)$ are adjacent to $n-1$ subgraphs in $A_{n,2}$.

\begin{thm}\label{Th4}
Let $n \geq5$ be an integer. Then $smp^{k}(A_{n,2})=2n-4$.
\end{thm}
\begin{proof}
Since $smp^{k}(A_{n,2})\leq \delta(A_{n,2})=2n-4$, we only need to prove that $smp^{k}(A_{n,2})\geq 2n-4$. It suffices to prove that $A_{n,2}-F$ has a perfect integer $k$-matching or an almost perfect integer $k$-matching for any $F\subseteq V(A_{n,2})\cup E(A_{n,2})$ with that $|F|=2n-5$.

If $A_{n,2}-F$ is even, then $A_{n,2}-F$ has a perfect matching since $smp(A_{n,2})=2n-4$ by Lemma \ref{Thm:1.10}. It follows that $A_{n,2}-F$ has a perfect integer $k$-matching. Suppose that $A_{n,2}-F$ is odd. Then the number of odd $G_{i}$ among $G_{1}, G_{2},\cdots,G_{n}$ is odd. We distinguish the following cases to prove that $A_{n,2}-F$ has an almost perfect integer $k$-matching.

\noindent\textbf{Case 1.} $|F_{1}|\leq n-4$.

Then $|F_{i}|\leq n-4$ for any $i\geq 2$. By Lemma \ref{lem:3.10}, $G_{i}$ $(i\geq1)$ is Hamiltonian. If there is exactly one odd subgraph among $G_1,G_2,\cdots,G_n$, then $A_{n,2}-F$ has an almost perfect integer $k$-matching clearly. Suppose that there exist at least three odd subgraphs among $G_1,G_2,\cdots,G_n$.

\textbf{Claim 1.} There exists at least one pair of subgraphs with a cross edge in $A_{n,2}-F$ among any three subgraphs $G_i,G_j,G_k$, where $1\leq i<j<k\leq n$.

For any three subgraphs $H_i,H_j,H_k$, let
$$S_i=F\cap \{v\in V(H_i)|\ v \ is \ adjacent \ to\ both \ H_j \ and \ H_k\},$$
$$S_j=F\cap\{\ v\in V(H_j)|\ v \ is \ adjacent \ to\ both \ H_i \ and \ H_k, \ but \ v \ is\ not\ adjacent\ to\ S_i\},$$
$$S_k=F\cap\{v\in V(H_k)|\ v \ is \ adjacent \ to\ both \ H_i \ and \ H_j, \ but \ v \ is\ adjacent\ to\ neither\ S_i \ nor \ S_j\}.$$
Suppose that $S_i=\{l_1i,l_2i,\cdots,l_pi\}$, where $\{l_1,l_2,\cdots,l_p\}\subseteq \{1,2,\cdots,n\}-\{i,j,k\}$.
Then $$S_j\subseteq\left\{lj|l\in \{1,2,\cdots,n\}-\{i,j,k\}-\{l_1,l_2,\cdots,l_p\}\right\}.$$
Let $S_j=\{l_{p+1}j,l_{p+2}j,\cdots,l_{p+p_1}j\}$. Then
$$S_k\subseteq\left\{lk|l\in \{1,2,\cdots,n\}-\{i,j,k\}-\{l_1,l_2,\cdots,l_p\}-\{l_{p+1},l_{p+2},\cdots,l_{p+p_1}\} \right\}.$$
Thus $|S_i|+|S_j|+|S_k|\leq n-3$. Let $m=|S_i|+|S_j|+|S_k|$. Then there are $3(n-2)-2m$ cross edges among $H_i-S_i,\ H_j-S_j$, $H_k-S_k$. And each vertex in $(F_i-S_i)\cup(F_j-S_j)\cup (F_k-S_k)$ is adjacent to exactly one subgraph among $H_i-S_i,\ H_j-S_j, \ H_k-S_k$. Let $S=S_i\cup S_j\cup S_k$. Then deleting an element in $F-S$ destroys at most one cross edge among $H_i-S_i,\ H_j-S_j$, $H_k-S_k$. Since $|F-S|=2n-5-m$, there are at most $2n-5-m$ cross edges among $H_i-S_i,\ H_j-S_j, H_k-S_k$ destroyed after deleting all elements in $F-S$. Since $(3(n-2)-2m)-(2n-5-m)=n-m-1\geq n-(n-3)-1=2$, Claim 1 is true.

By Claim 1, we can obtain an almost perfect integer $k$-matching of $A_{n,2}-F$ since each $G_i$ ($i\geq1$) is Hamiltonian.

\noindent\textbf{Case 2.} $|F_{1}|=n-3$.

Then $|F_{i}|\leq n-3$ for any $i\geq2$, $|F-F_1|=n-2$ and $G_{1}$ has a Hamiltonian path, say $P$. Let $x_1$ and $y_1$ be two end-vertices of $P$.

Firstly, suppose that $|F_{2}|\leq n-4$. By Lemma \ref{lem:3.10}, $G_{i}(i\geq 2)$ is Hamiltonian. If $G_{1}$ is even, then $A_{n,2}-F$ has an almost perfect integer $k$-matching since there exists at most a pair of subgraphs among $G_{2}, G_3, \cdots,G_n$ without cross edges.
Suppose that $G_{1}$ is odd. Since $|F-F_1|=n-2$ and $x_1$ and $y_1$ are adjacent to $n-1$ subgraphs $H_2,H_3,\cdots,H_n$, without loss of generality, we can assume that $w\in V(H_{n})$ such that $x_1w\in E(A_{n,2}-F)$ and $|F_n|=0$. Then the subgraph induced by $V(G_1)\cup \{w\}$ has a perfect integer $k$-matching and $G_n=H_{n}-\{w\}$ is Hamiltonian by Lemma \ref{lem:3.10}. Since $|(F-F_1)\cup \{w\}|=n-1$ and $|E_{A_{n,2}}(V(H_i),V(H_j))|=n-2$ ($2\leq i<j\leq n$), there exists at most a pair of subgraphs in $A_{n,2}-F$ among $G_{2}, G_3, \cdots,G_n$ without cross edges. Thus we can obtain an almost perfect integer $k$-matching of $A_{n,2}-F$.

Secondly, suppose that $|F_{2}|=n-3$. Then $|F_{3}|\leq1$ and $|F_{i}|=0(i\geq4)$. By Lemma \ref{lem:3.10}, $G_{i}$ $(i\geq3)$ is Hamiltonian and $G_{2}$ has a Hamiltonian path, say $Q$. Let $x_2$ and $y_2$ be two end-vertices of $Q$.
Since $|F-F_1-F_2|=1$ and $|F_i|=0$ ($i\geq4$), without loss of generality, we can assume that $x_1$ is adjacent to $H_4$ and $x_2$ is adjacent to $H_5$ in $A_{n,2}-F$. Let $x_1w_1$, $x_2w_2\in E(A_{n,2}-F)$, where $w_1\in V(H_4)$ and $w_2\in V(H_5)$. Let $M_1$ and $M_2$ be the maximum matching of $P$ and $Q$ such that the $M_1$-unsaturated and $M_2$-unsaturated vertices are $x_1$ and $x_2$ (if they exist), respectively. Let $X_1=\{w_1\}$ if $G_1$ is odd and $X_2=\{w_2\}$ if $G_2$ is odd. Otherwise, let $X_1=\emptyset$ and $X_2=\emptyset$. Thus the subgraph induced by $V(G_1)\cup X_1$ and $V(G_2)\cup X_2$ have a perfect matching, respectively. By Lemma \ref{lem:3.10}, $G_4=H_{4}-X_1$ and $G_5=H_{5}-X_2$ are Hamiltonian. Since $|F-F_1-F_2|=1$ and $|\{X_1\cup X_2\}|\leq2$, each pair of subgraphs among $G_3,G_{4},\cdots,G_n$ has at least one cross edge. Thus we can obtain an almost perfect integer $k$-matching of $A_{n,2}-F$.

\noindent\textbf{Case 3.} $n-1\leq |F_{1}|\leq 2n-5$.

Then $|F-F_1|\leq n-4$. Let $M$ be a maximum matching of $G_1$ and $X$ the set of $M$-unsaturated vertices of $G_1$. Since $H_1\cong K_{n-1}$ and $n
\geq 5$, $|X|\leq n-2$. Let $X=\{v_1,v_2,\cdots,v_l\}$, where $0\leq l\leq n-2$. Let $G_0=(X,Y)$ be a bipartite graph with that $Y=\{G_3,G_4,\cdots,G_n\}$ and $v_iG_j\in E(G_0)$ if and only if $v_i$ is adjacent to $G_j$ $(j\geq 3)$ in $A_{n,2}-F$.

\textbf{Claim 2.} $G_0$ contains a matching that saturates every vertex in $X$.

By the Hall's theorem, we only need to prove that $|N_{G_0}(S)|\geq |S|$ for any $S\subseteq X$.
Since $|F-F_{1}|\leq n-4$ and each vertex has $n-2$ outer neighbors in $A_{n,2}$, each vertex in $X$ has at least one outer neighbor in  $A_{n,2}-F-G_2$. Thus $|N_{G_0}(S)|\geq 1=|S|$ for any $S$ with that $|S|=1$.
Since each pair of vertices in $V(A_{n,2})$ are adjacent to $n-1$ different subgraphs $H_i$ and $|F-F_{1}|\leq n-4$, each pair of vertices in $X$ are adjacent to at least $(n-1)-1-(n-4)$ subgraphs among $G_3,G_4,\cdots,G_n$. So $|N_{G_0}(S)|\geq 2=|S|$ for any $S$ with that $|S|=2$.
Let $3\leq|S|\leq n-3$. Since $S$ has at least $|S|-1$ outer neighbors in $V(H_i)$ for each $i\geq3$, there are at most $\frac{n-4}{|S|-1}$ subgraphs $G_i$ ($i\geq3$) not adjacent to $S$ in $A_{n,2}-F$. It follows that $|N_{G_0}(S)|\geq (n-2)-\frac{n-4}{|S|-1}\geq |S|$.
Suppose that $|S|=n-2$. Since $|E_{A_{n,2}}(S, V(H_i))|\geq n-3$ for each $i\geq3$ and $|F-F_1|\leq n-4$, each subgraph $G_i\in Y$ is adjacent to at least one vertex in $X$ of $A_{n,2}-F$. Then $|N_{G_0}(S)|=|Y|=n-2=|S|$. Hence Claim 2 is true.

By Claim 2, we can assume that $v_1',v_2',\cdots,v_l'$ are the outer neighbor of $v_1,v_2,\cdots,v_l$, respectively, such that $v_1',v_2',\cdots,v_l'$ belong to $l$ distinct subgraphs among $G_3,G_4,\cdots,G_n$. Thus the subgraph induced by $V(G_1)\cup \{v_1',v_2',\cdots,v_l'\}$ has a perfect integer $k$-matching.
Let $F_i'=F_i\cup \left(V(G_i)\cap \{v_1',v_2',\cdots,v_l'\}\right)$ and $G_i'=H_i-F_i'$, where $i\geq2$. Then $F_2'=F_2$ and $|F_i'|\leq |F_2|+1$ for any $i\geq 3$.

If $|F_{2}|=0$, then $|F_i|=0$ for any $i\geq3$. Thus $|F_i'|\leq 1$ ($i\geq 3$). By Lemma \ref{lem:3.10}, $G_i'$ ($i\geq2$) is Hamiltonian. Hence $A_{n,2}-F$ has an almost perfect integer $k$-matching.
Suppose that $|F_2|\geq1$. Then $|F_i|\leq |F-F_{1}-F_2|\leq n-4-1=n-5$ for any $i\geq3$. So $|F_i'|\leq n-5+1=n-4$. Since $|F_2'|=|F_2|\leq |F-F_1|\leq n-4$, $G_{i}'(i\geq2)$ is Hamiltonian by Lemma \ref{lem:3.10}. Since $|F-F_{1}|\leq n-4$ and $|E_{A_{n,2}}(V(H_{i}),V(H_{j}))|=n-2$, there are at most one pair of subgraphs among $G_2',G_3',\cdots,G_n'$ without cross edges. Thus the subgraph induced by $V(G_{2}')\cup V(G_{3}')\cup\cdots\cup V(G_{n}')$ in $A_{n,2}-F$ has an almost perfect integer $k$-matching. Hence $A_{n,2}-F$ has an almost perfect integer $k$-matching.

\noindent\textbf{Case 4.} $|F_{1}|=n-2$.

Let $M$ be a maximum matching of $G_1$ and $X$ the set of $M$-unsaturated vertices of $G_1$. Then $|X|\leq n-3$ since $H_1\cong K_{n-1}$ and $|F_{1}|=n-2$. Let $X=\{v_1,v_2,\cdots,v_t\}$, where $0\leq t\leq n-3$. Let $G_0=(X,Y)$ be a bipartite graph with that $Y=\{G_2,G_3,\cdots,G_n\}$ and $v_iG_j\in E(G_0)$ if and only if $v_i$ is adjacent to $G_j$ in $A_{n,2}-F$.

\textbf{Claim 3.} $G_0$ contains a matching that saturates every vertex in $X$.

Similar to the method of the proof of Claim 2, we can prove that Claim 3 is true.

By Claim 3, we can assume that $v_1',v_2',\cdots,v_t'$ are the outer neighbor of $v_1,v_2,\cdots,v_t$, respectively, such that $v_1',v_2',\cdots,v_t'$ belong to $t$ distinct subgraphs among $G_2,G_3,\cdots,G_n$. Thus the subgraph induced by $V(G_1)\cup \{v_1',v_2',\cdots,v_t'\}$ has a perfect integer $k$-matching.
Let $F_i'=F_i\cup \left(V(G_i)\cap \{v_1',v_2',\cdots,v_t'\}\right)$ and $G_i'=H_i-F_i'$, where $i\geq2$. Then $F_i'=F_i$ or $F_i'=F_i\cup \{v_j'\}$. Note that $|F_2|\leq |F-F_1|=n-3$.

\noindent\textbf{Subcase 4.1.} $|F_{2}|\leq n-5$. Then $|F_{i}|\leq n-5$ ($i\geq3$). Thus $|F_{i}'|\leq n-4$ for any $i\geq2$. By Lemma \ref{lem:3.10}, $G_{i}'(i\geq2)$ is Hamiltonian. Since $|F-F_1|=n-3$, there exist at most two pair of subgraphs among $G_2',G_3',\cdots,G_n'$ without cross edges. Hence the subgraph induced by $V(G_{2}')\cup\cdots\cup V(G_{n}')$ in $A_{n,2}-F$ has an almost perfect integer $k$-matching. Thus $A_{n,2}-F$ has an almost perfect integer $k$-matching.

\noindent\textbf{Subcase 4.2.} $|F_{2}|=n-3$.
Then $F=F_1\cup F_2$. Thus $|F_2'|\leq n-2$ and $|F_i'|\leq1$ $(i\geq 3)$. By Lemma \ref{lem:3.10}, $G_{i}'\ (i\geq 3)$ is Hamiltonian and $G_2'$ contains a spanning subgraph which consist of two paths, say $P$ and $Q$, respectively. Let $x_1$ and $y_1$ are the two end-vertices of $P$ and $x_2$ and $y_2$ are the two end-vertices of $Q$. Since $t\leq n-3$, there is a subgraph among $H_3, H_4,\cdots,H_n$, say $H_3$ such that $V(H_3)\cap \{v_1',v_2',\cdots,v_t'\}=\emptyset$. Since $F-F_1-F_2=\emptyset$, $G_3'=H_3$. Since $\{x_1,y_1,x_2,y_2\}\subseteq V(H_2)$, there is at most one in $\{x_1,y_1,x_2,y_2\}$ not adjacent to $H_3$. Without loss of generality, we can assume that $w_1, w_2\in V(H_3)$ are the outer neighbor of $x_1$ and $x_2$, respectively.
Let $M$ be a maximum matching of $G_2'$ and $X$ the set of $M$-unsaturated vertices of $G_2'$ such that $X\subseteq \{x_1,x_2\}$ (it is possible that $X=\emptyset$). Let $X'$ be the outer neighbor of $X$ such that $X'\subseteq \{w_1,w_2\}$. Then the subgraph induced by $V(G_2')\cup X'$ has a perfect matching and $H_3-X'=K_2$ (in this case, $n=5$) or $H_3-X'$ is Hamiltonian by Lemma \ref{lem:3.10}. Clearly, each pair of subgraphs among $G_4',G_5',\cdots,G_n'$ has cross edges. Hence we can obtain an almost perfect integer $k$-matching of $A_{n,2}-F$.

\noindent\textbf{Subcase 4.3.} $|F_{2}|=n-4$. Then $|F_{3}|\leq 1$ and $|F_{i}|=0$ $(i\geq4)$.
Suppose that $|F_{2}'|=n-3$. Then we can assume that $v_1'\in V(G_2)$. By Lemma \ref{lem:3.10}, $G_2'$ has a Hamiltonian path, say $P$. Let $x$ and $y$ be two end-vertices of $P$. If $P$ is even, then $A_{n,2}-F$ has an almost perfect matching. Suppose that $P$ is odd. Since $x$ and $y$ are adjacent to $n-2$ subgraphs $H_3,H_4,\cdots, H_n$ and $|\{v_2',v_3',\cdots,v_t'\}|\leq n-4$ and $|F-F_1-F_2|=1$, there is a subgraph, say $G_i'$, adjacent to $x$ or $y$ in $A_{n,2}-F-\{v_2',v_3',\cdots,v_t'\}$ such that $|F_i|=0$. Suppose that $w\in V(G_i')$ and $w$ is an outer neighbor of $x$ or $y$ in $A_{n,2}-F-\{v_2',v_3',\cdots,v_t'\}$. Then the subgraph induced by $V(G_2')\cup \{w\}$ has a perfect matching.
By Lemma \ref{lem:3.10}, $G_i'-w=K_2$ (in this case, $n=5$ and $|F_i'|=1$) or $G_i'-w$ is Hamiltonian and $G_{j}'(j\geq4, j\neq i)$ is Hamiltonian.
If $i=3$, then $G_3'-w=K_2$ or $G_3'-w$ is Hamiltonian and $G_{j}'(j\geq4)$ is Hamiltonian. Thus we can obtain an almost perfect integer $k$-matching of $A_{n,2}-F$.

Now we consider the case that $i\neq 3$. For $G_3'$, we have that $|F_3'|\leq |F_3|+1\leq2$. If $|F_3'|=1$ or $n\geq 6$, then $G_3'$ is Hamiltonian by Lemma \ref{lem:3.10}. Suppose that $|F_3'|=2$ and $n=5$. Then $G_3'= K_2$ (in this case $F_3'\subset V(H_3)$) or $G_3'= K_3-e$ (in this case, $F_3'$ consist of one edge and one vertex). When $G_3'$ is Hamiltonian or $G_3'= K_2$, we can obtain an almost perfect integer $k$-matching of $A_{n,2}-F$ since $G_{j}'(j\geq4,j\neq i)$ is Hamiltonian and there are at most two pair of subgraphs among $G_3', \cdots,G_i'-w,\cdots,G_n'$ without cross edges.
Suppose that $G_3'= K_3-e$, that is, $G_3'$ is a path. Then $n=5$, $t=2$, $|F_3'|=2$ and $F=F_1\cup F_2\cup F_3$. Thus we can assume that $v_2'\in V(G_3)$. Let $u$ and $v$ be two end-vertices of $G_3'$. Without loss of generality, we can assume that $i=4$. Then $u$ or $v$ has an outer neighbor $z\in V(H_5)$. Thus the subgraph induced by $V(G_3')\cup \{z\}$ has a perfect matching and $H_5-z$ is Hamiltonian by Lemma \ref{lem:3.10}. Hence we can obtain an almost perfect integer $k$-matching of $A_{n,2}-F$.

Finally, we consider the case that $|F_{2}'|=|F_{2}|=n-4$. Then every $G_j'$ ($j\neq1,3$) is Hamiltonian. By the same discussion as above, $G_3'=K_2$ or $G_3'= K_3-e$ or $G_3'$ is Hamiltonian. By the same reason as above, $A_{n,2}-F$ has an almost perfect integer $k$-matching.
\end{proof}

\section{Acknowledgments}

This work This work is supported by the Natural Science Foundation of Guangdong (No.2021A 1515012045) and the National Natural Science Foundation of China (No.12161073).

\bibliography{bibfile1}

 % So we pose the following problem.

%\begin{prob}
%Does $va_{eq}^{*}(G)\leq 3$ hold for every planar graph $G$?
%\end{prob}

%In \cite{WZL2013}, Wu et al. proved that $va_{eq}^{*}(G)\leq3$ if $g(G)\geq5$ and $va_{eq}^{*}(G)\leq2$ if  $g(G)\geq6$. In \cite{Z2014}, Zhang proved that $va_{eq}^{*}(G)\leq3$ if all cycles of length at most 4 are independent and there are no 3-cycles and adjacent 4-cycles in $G$.

%\vspace{0.3 cm}

\end{document}